\documentclass{amsart}
\usepackage{amssymb}
\usepackage{amsthm}
\usepackage{mathtools}
\usepackage{environ}
\usepackage{enumerate}
\usepackage{microtype}
\usepackage{enumitem}
\usepackage[utf8]{inputenc}

\theoremstyle{plain}
\newtheorem{theorem}{Theorem}[section]
\newtheorem{lemma}[theorem]{Lemma}

\newtheorem{corollary}[theorem]{Corollary}

\theoremstyle{definition}
\newtheorem{definition}[theorem]{Definition}
\newtheorem{example}[theorem]{Example}

\theoremstyle{remark}
\newtheorem*{remark}{Remark}

\newcommand{\bd}{\partial}

\newcommand{\C}{\mathbb{C}}
\newcommand{\R}{\mathbb{R}}
\newcommand{\N}{\mathbb{N}}
\newcommand{\psh}{\mathcal{PSH}}
\newcommand{\usc}{\mathcal{USC}}

\newcommand{\suchthat}{\mathrel{;}}

\title[Plurisubharmonic functions with discontinuous boundary behavior]{Plurisubharmonic functions with \\ discontinuous boundary behavior}
\author{Mårten Nilsson}
\address{Center for Mathematical Sciences\\
  Lund University\\
  Box 118, SE-221 00 Lund, Sweden}
\email{marten.nilsson@math.lth.se}


\begin{document}

\maketitle
\begin{abstract}
   We study the Dirichlet problem for the complex Monge--Ampère operator with bounded, discontinuous boundary data. If the set of discontinuities is b-pluripolar and the domain is B-regular, we are able to prove existence, uniqueness and some regularity estimates for a large class of complex Monge--Ampère measures. This result is optimal in the unit disk, as boundary functions with b-pluripolar discontinuity then coincide with functions that are continuous almost everywhere. We also show that neither of these properties of the boundary function -- being continuous almost everywhere or having discontinuities forming a b-pluripolar set -- are necessary conditions in order to establish uniqueness and continuity of the solution in higher dimensions, if one additionally requires that lower limits of the solution coincides with the lower limits of the boundary function. In particular, there are situations where it is enough to prescribe the limit behavior at a set of arbitrarily small Lebesgue measure.
\end{abstract}
\section{Introduction}
 A central result in potential theory is the solution of the generalized Dirichlet problem, namely that for a general domain \(\Omega \subset \C_\infty\) with non-polar boundary and \(\phi: \partial \Omega \rightarrow \R\) continuous n.e.\ (nearly everywhere, i.e.\ outside a polar set), there exists a unique bounded harmonic function \(h\) such that 
\[\lim_{z \rightarrow \zeta \in \partial \Omega}h(z) = \phi(\zeta)\quad\text{ n.e.}\]  
The real power of this theorem lies in the generality of the domain, and its proof relies on the extended maximum principle. For regular domains, such as the unit disk, stronger results are available due to integral representations. In particular, the Herglotz--Riesz representation theorem implies that bounded harmonic functions on the unit disk are in one-to-one correspondence with  essentially bounded functions on the circle, and one may show that the boundary values are attained at all points where the boundary data is continuous. See Ransford~\cite[Corollary~4.2.6]{ransford} and Garnett and Marshall~\cite[Corollary~I.2.5]{garnett} for more details.

Few attempts have been made to investigate to which degree these results persist in the plurisubharmonic setting, although we should mention that specific examples of plurisubharmonic functions with boundary discontinuities were considered already by Bedford~\cite{bedford}. The main obstruction arising in higher dimensions is that due to nonlinearity of the complex Monge--Ampère operator, it is no longer enough to appeal to a maximum principle to establish uniqueness to the Dirichlet problem, and we do not have a Poisson integral at our disposal. Instead, one needs to use the comparison principle, with the major drawback that pluripolar sets cannot be neglected. 

However, in a recent paper, Rashkovskii~\cite{rashkovskii} observed that it is indeed possible to formulate such a comparison principle, only referring to points on the boundary outside a pluripolar set. In Section~2, we modify his proof slightly to also encompass \textit{b-pluripolar sets}, a notion recently introduced by Djire and Wiegerinck \cite{djire}. We say that a set \(F \subset \partial \Omega\) is b-pluripolar if there exists \(v \in \psh(\Omega)\) such that \(v\leq 0, v \not \equiv -\infty\) and \(v^* = - \infty\) on \(F\). Our generalization of Rashkovskii's comparison principle provides a uniqueness argument for the inhomogeneous Dirichlet problem for the complex Monge--Ampère operator,
    \[\begin{cases}
    u \in \psh(\Omega)\cap L^\infty(\Omega)  \\
    (dd^cu)^n=\mu \\
        \lim_{\Omega \ni \zeta \rightarrow z_0} u(\zeta) = \phi(z_0),\quad\forall z_0 \in \bd\Omega \setminus E_\phi,
    \end{cases}\]
where \(\Omega\) is B-regular and \(\phi: \bd \Omega \rightarrow \R\) is a bounded function, continuous outside a b-pluripolar set \(E_\phi\). Conversely, we show that the uniqueness of such Dirichlet problems implies that $E_\phi$ is b-pluripolar. We also show that for a large class of densities, one may estimate the extent of discontinuities in the interior, and we provide an example which shows that this estimate is sharp.

 Compellingly, the above result provides an alternative description of negligible sets with respect to harmonic measure. In particular, b-pluripolar discontinuity sets and discontinuity sets of Lebesgue measure zero coincide in the case of the unit circle \(\bd D\) (see also the second remark to Theorem~\ref{maintheorem} below). On the other hand, given \(A \subset \bd D\) and a fixed bounded function \(\phi: \bd D \rightarrow \R\), continuous outside \(A\), the following four statements are equivalent:
\begin{enumerate}[label=(\roman*)]
    \item \(A\) is b-pluripolar.
    \item \(A\) has Lebesgue measure zero.
        \item There exists a unique bounded harmonic function \(h_\phi\) such that
    \begin{align*}
                \lim_{D \ni \zeta \rightarrow z_0} h_\phi(\zeta) &= \phi(z_0),\quad\forall z_0 \in \bd D \setminus A.
    \end{align*}
    \item There exists a unique bounded harmonic function \(h_\phi\) such that
    \begin{alignat*}{2}
                \lim_{D \ni \zeta \rightarrow z_0} h_\phi(\zeta) &= \phi(z_0), && \forall z_0 \in \bd D \setminus A, \\
                \liminf_{D \ni \zeta \rightarrow z_0} h_\phi(\zeta) &= \liminf_{\partial D\setminus A \ni \zeta \rightarrow z_0}\phi(\zeta), \quad && \forall z_0 \in A. 
    \end{alignat*}

\end{enumerate}
 This leads us to the following questions: Is it possible to find B-regular domains in \(\C^n, n>1\) and boundary data for which the Dirichlet problem (with the addition of lower limits) is uniquely solvable, where the discontinuities form a set that
\begin{enumerate}[label=\alph*)]
    \item is not b-pluripolar, and
    \item has positive Lebesgue measure?
\end{enumerate}
With the goal of providing affirmative answers to these questions, we introduce, in Section~3, a class of functions on Reinhardt domains for which several Perron--Bremermann constructions are particularly well-behaved, ultimately due to an approximation result of Wiegerinck and Fornæss~\cite{fornaess}. An interesting corollary of these considerations is that the classic theorem due to J.B.~Walsh~\cite{walsh}, which in modern terminology says that on B-regular domains,
 \[
\varphi \text{ uniformly continuous on }\Omega  \implies P(\varphi) \in C(\overline\Omega),
 \]
has a counterpart valid on all bounded Reinhardt domains. Here, \(P(\varphi)\) denotes the Perron--Bremermann envelope
 \[
P(\varphi) := \sup \{u(z) \suchthat u \in \psh(\Omega), u^* \leq \varphi\}
 \]
of a function \(\varphi: \Omega \rightarrow \R\). Note that the regularization enables us to use the same notation when \(\varphi\) instead is defined on the boundary of \(\Omega\). 

Finally, in Section 4, we show that on the unit ball with torically invariant characteristic functions as boundary data, the Dirichlet problem does indeed for a large class of measures admit unique solutions, continuous in the interior. Here, the uniqueness argument is provided by Djire and Wiegerinck's partial answer~\cite[Theorem 2.11]{djire} to Sadullaev's question concerning when the upper semicontinuous regularization of a variety of boundary extremal functions coincide~\cite{sadu}. As a result, we are able to show that (iv) does not imply (i), (ii) or (iii) in the plurisubharmonic setting.

The author would like to thank Frank Wikström for helpful discussions, comments and suggestions.
\section{Boundary values with b-pluripolar discontinuity}
In this section we consider the Dirichlet problem for the complex Monge--Ampère equation under the assumption that the discontinuities in the boundary data form a b-pluripolar set. In this case, uniqueness follows from the following extended version of the domination principle.  
\begin{lemma}\label{rash} Let \(u,v \in \psh(\Omega)\cap L^\infty(\Omega), \Omega \Subset \C^n\) and suppose that
\[
\limsup_{z \rightarrow \zeta}(u(z) - v(z)) \leq 0 \quad \forall \zeta \in \partial \Omega \setminus F,
\]
where \(F \subset \partial \Omega\) is b-pluripolar. If \((dd^c v)^n \leq (dd^c u)^n\), then \(u \leq v\) on \(\Omega\). In particular, if \(\lim_{z \rightarrow \zeta}(u(z) - v(z)) = 0\) for all \(\zeta \in \partial \Omega \setminus F\) and \((dd^c v)^n = (dd^c u)^n\), then \(u = v\).
\end{lemma}
\begin{proof}
As in the proof of the domination principle~\cite{guedj}, it is enough to establish the corresponding comparison principle, i.e.\ that our assumptions imply that
\[
\int_{v < u} (dd^cv)^n \geq  \int_{v < u} (dd^cu)^n.
\]
We reason as follows: Since \(F\) is b-pluripolar, we may find \(\phi \in \psh(\Omega)\) satisfying \(\phi\leq 0\) and \(\phi^* = - \infty\) on \(F\), such that
\[
P_\phi := \{z \in \Omega \suchthat \phi = - \infty\}
\]
is a pluripolar set. Replacing \(\delta \psi\) by \(\delta \max\{\phi, -\frac{1}{\delta^2}\}\) in the proof of \cite[Lemma 3.5]{rashkovskii}, we construct 
\[
u_\delta := u + \delta (\max\{\phi, - \frac{1}{\delta^2}\}-1)
\]
and notice that
\[
\limsup_{z \rightarrow \zeta}(u_\delta(z) - v(z)) \leq - \delta \quad \forall \zeta \in \partial \Omega 
\]
for \(\delta\) small enough. It follows from the comparison principle that 
\[
\int_{v < u_\delta} (dd^cv)^n \geq \int_{v < u_\delta} (dd^cu_\delta)^n \geq \int_{v < u_\delta} (dd^cu)^n,
\]
and letting \(\delta \rightarrow 0\), we conclude that
\[
\int_{v < u} (dd^cv)^n \geq  \int_{v < u} (dd^cu)^n,
\]
since \((dd^cu)^n(P_\phi)=(dd^cv)^n(P_\phi) = 0\). 
\end{proof}
The argument needed to provide existence remains valid for a wide class of complex Monge--Ampère measures. We use the following terminology, introduced in \cite{nilsson}. 
\begin{definition}
A measure \(\mu\) is said to be \textit{compliant} if
\[
    \begin{cases}
    u \in \psh(\Omega)\cap L^\infty(\Omega)  \\
    (dd^cu)^n=\mu \\
        \lim_{\Omega \ni \zeta \rightarrow z_0 \in \bd \Omega } u(\zeta) = \phi(z_0), \quad \forall z_0 \in \bd \Omega
    \end{cases}
\]
has a unique solution for every \(\phi \in C(\bd \Omega)\), where \(\Omega\) is a bounded domain. Furthermore, if the solution is always continuous, we say that \(\mu\) is \textit{continuously compliant}.
\end{definition}
\begin{remark}
Note that the existence of compliant measures implies that \(\Omega\) is B-regular, i.e.\ every continuous function \(\phi\) on \(\bd \Omega\) may be extended to a continuous, plurisubharmonic function in the interior. This follows from the fact that a compliant measure provides us with a plurisubharmonic function \(u_\phi\) such that
\[
\phi(z_0)=\liminf_{\Omega \ni \zeta \rightarrow z_0 \in \bd \Omega } u_\phi(\zeta) \leq \liminf_{\Omega \ni \zeta \rightarrow z_0 \in \bd \Omega } P(\phi)(\zeta),
\]
forcing \(P(\phi)\) to satisfy
\[
\lim_{\Omega \ni \zeta \rightarrow z_0 \in \bd \Omega }  P(\phi)(\zeta) = \phi(z_0).
\]
One may then apply \cite[Lemma~1]{walsh} to conclude that \(P(\phi)\) is a plurisubharmonic extension of \(\phi\), continuous in the interior. 
\end{remark}
When \(\mu\) is continuously compliant, it is possible to estimate at which points the solution might be discontinuous. This estimate will be given in terms of the defining family
\[
\mathcal{F}_{E_\phi}:=\{u\in \psh(\Omega) \suchthat u\not \equiv -\infty, u<0, u^*\mid_{E_\phi} = -\infty \}
\]
 for the \textit{b-pluripolar hull}
\[
\hat E_\phi:=\{z \in \bar \Omega \suchthat \forall u\in \mathcal{F}_{E_\phi}, u^*(z) = -\infty\}
\]
of the b-pluripolar set \(E_\phi\) of discontinuities on the boundary. 

We are now ready to formulate and prove our main result.
\begin{theorem}\label{maintheorem}
Let \(\mu\) be a compliant measure on a B-regular domain \(\Omega\), and let \(\phi: \partial \Omega \rightarrow \R\) be a bounded function, continuous outside \(E_\phi\). Then the Dirichlet problem
    \[\begin{cases}
    u \in \psh(\Omega)\cap L^\infty(\Omega)  \\
    (dd^cu)^n=\mu \\
        \lim_{\Omega \ni \zeta \rightarrow z_0} u(\zeta) = \phi(z_0),\quad\forall z_0 \in \bd\Omega \setminus E_\phi
    \end{cases}\]
has a unique solution if and only if \(E_\phi\) is b-pluripolar. If \(\mu\) additionally is continuously compliant, then the set of discontinuities is necessarily a subset of \[\bigcap_{u \in \mathcal{F}_{E_\phi}}\{u_*(z) = -\infty\}.\]
\end{theorem}
\begin{proof}
First assume that \(E_\phi\) is b-pluripolar. It is straightforward to show that 
\[P(\phi^*, \mu) := \sup \{u(z) \suchthat u \in \psh(\Omega), (dd^cu)^n \geq \mu, u^*\leq \phi^*\}
\]
satisfies \((dd^cP(\phi^*, \mu))^n=\mu\) using a standard balayage argument, and so by the extended domination principle, \(P(\phi^*, \mu)\) uniquely solves the Dirichlet problem provided that the boundary values are attained outside \(E_\phi\). To show that this is indeed the case, we adapt the proof of \cite[Theorem~4.1.5]{ransford}. Suppose without loss of generality that \(M > \phi > 0\), pick \(\zeta_0 \in \bd \Omega\) such that \(\phi\) is continuous at \(z_0\) and let \(N_0\) be a neighborhood of \(\zeta_0\) such that
\[
\zeta \in \partial \Omega \cap \bar N_0 \implies |\phi(\zeta) - \phi(\zeta_0)| < \varepsilon.
\]
Pick \(\psi \in C(\partial \Omega)\) such that \(\psi \leq 0\) on \(\partial \Omega\), only equal to zero at \(\zeta_0\), and let \(u_\psi \in \psh(\Omega) \cap L^\infty(\Omega)\) satisfy
\[
    \begin{cases}
    u_\psi \in \psh(\Omega)\cap L^\infty(\Omega)  \\
    (dd^cu_\psi)^n=\mu \\
        \lim_{\Omega \ni \zeta \rightarrow z_0} u_\psi(\zeta) = \psi(z_0), \quad\forall z_0 \in \bd\Omega.
    \end{cases}
\]
Multiplying \(\psi\) by a large constant if necessary, we may assume that \(u_\psi < -1\) on \(\Omega\setminus N_0\). Now note that
\[
u := \phi(\zeta_0) - \varepsilon + (M+1+ \phi(\zeta_0))u_\psi
\]
satisfies
\[
\lim_{z \rightarrow \zeta\in \partial \Omega} u(z) \leq \phi(\zeta), \quad \lim_{z \rightarrow \zeta_0} u(z) = \phi(\zeta_0) - \varepsilon, \quad (dd^cu)^n \geq \mu, 
\]
and letting \(\varepsilon \rightarrow 0\), we conclude that \(\lim_{z \rightarrow \zeta_0} P(\phi^*, \mu)(z) = \phi(\zeta_0)\).

We will now show that uniqueness of the solution implies that \(E_\phi\) is b-pluripolar. Without loss of generality, we may assume that $K_1<\phi \leq 0$ on $\partial \Omega \setminus E_\phi$ and the unique solution $u_{\mu, \phi}$ satisfies $K_2<u_{\mu, \phi} \leq 0$ on $\Omega$. Now consider the lower semicontinuous functions
\[
\phi_m(\zeta) =     
\begin{cases}
    \phi(\zeta) \quad &\zeta \in \partial\Omega \setminus E_\phi \\
    2^mK_1 \quad &\zeta \in E_\phi,
    \end{cases}
\]
and associate to $\phi_m$ a sequence $\phi_{m,n}$ of continuous functions such that $\phi_{m,n} \nearrow \phi_{m}$. From the families
\[
\mathcal{F}_{\phi_{m,n}}:= \{u \in \psh(\Omega)  \suchthat (dd^cu)^n \geq \mu, u^*\leq \phi_{m,n}\} 
\]
we then construct
\[u_{m,n}(z) := \sup \{u(z) \suchthat u \in \mathcal{F}_{\phi_{m,n}}\},\]
and note that by uniqueness and by monotonicity of the complex Monge--Ampère operator,  $u_{m,n} \nearrow u_{\mu, \phi}$ outside a pluripolar set $P_m$ as $n \rightarrow \infty$.  Since $\cup P_m$ is pluripolar as well, there exists a point $z_0 \in \Omega \setminus \cup P_m$ and $n(m)$ such that $u_{m,n(m)}(z_0) > K_2$ for all $m \in \N$, which implies that the series
\[\tilde u(z) = \sum_{m=1}^\infty \frac{1}{2^m} u_{m,n(m)}(z)
\]
converges to plurisubharmonic function. By construction, $\limsup_{z \rightarrow E_\phi} \tilde u(z) \leq mK_1$ for all $m$, which shows that \(E_\phi\) is b-pluripolar.

In order to prove the last statement, fix \(v \in \mathcal{F}_{E_\phi}\) and assume that \(\mu\) is continuously compliant. Then 
\[
u_\varepsilon := \max\{P(\phi^*, \mu)+\varepsilon v, \min\{\inf \phi, \inf P(\phi^*, \mu) \} \}
\]
satisfies \(u_\varepsilon^* \leq \phi_*\) on \(\partial \Omega\), and by the Katětov–Tong insertion theorem~\cite{tong}, we may find \(\phi_\varepsilon \in C(\partial \Omega)\) such that \(\phi_\varepsilon \leq \phi_*\) and
\[
u_\varepsilon \leq P(\phi_\varepsilon, \mu)\leq P(\phi^*, \mu).
\]
Since \(u_\varepsilon \nearrow P(\phi^*, \mu)\) on the open set \(\{v_* \neq- \infty\}\) and \(P(\phi_\varepsilon, \mu) \in C(\bar \Omega)\), the claim of the theorem follows.
\end{proof}
\begin{remark}
In one complex dimension, this theorem reduces to the harmonic case as we may subtract a subharmonic function \(u\) with the properties
\begin{align*}
    dd^cu = \mu, \quad \lim_{z \rightarrow \bd \Omega}u(z)=0
\end{align*}
by the compliance of \(\mu\). 
\end{remark}
\begin{remark}
In the case of the unit disk, the set $E_\varphi$ is of Lebesgue measure zero if and only if it is b-pluripolar. This follows from a theorem of Fatou~\cite{fatou}, which says that we for any compact set $K \subset \partial D$ of Lebesgue measure zero may find a function $f$ such that $f: \partial D \rightarrow [-\infty, 0]$ continuously,  $f\in  L^1(\partial D)$ and 
\[
f(z)=-\infty \iff z \in K.
\]
Now note that since $E_\varphi$ is a $F_\sigma$ set, we may write \(E_\varphi\) as a countable union of compact sets $K_i$. Extending the corresponding $f_i$ to harmonic functions $h_i$ in the interior, we then construct
\[
h : = \sum_{i=1}^\infty c_i h_i,
\]
where $c_i>0$ are chosen such that the sum converges at some point in the interior. By Harnack's theorem, the sum converges everywhere to a harmonic function, which shows that $E_\varphi$ is b-pluripolar. 
\end{remark}
The following example shows that the set of discontinuities may be nonempty, and may even coincide with \(\cap_{u \in \mathcal{F}_{E_\varphi}}\{u_*(z) = -\infty\}\).
\begin{example}
Consider the plurisubharmonic function
		\[
		\tilde u(z_1,z_2) : = \max\{\sum_{k= 1}^\infty 2^{-k}\log|z_1 - 2^{-k}|, -1\}
		\]
  restricted to the unit ball $\mathbb{B} \subset \mathbb{C}^2$. Clearly, $\tilde u$ satisfies $(dd^c\tilde u)^n = 0$, is discontinuous on \(\{z_1 = 0\}\cap \mathbb{B}\), and extends continuously to the boundary outside the b-pluripolar set 
  \[
  \tilde E: = E_{\tilde u \mid_{\partial \mathbb{B}}}= \{z_1 = 0\} \cap \partial \mathbb{B}.
  \]
  Since the zero measure is continuously compliant on B-regular domains, $\tilde u$ uniquely solves a Dirichlet problem satisfying the conditions of Theorem~\ref{maintheorem}. Precomposing with the analytic disk 
  \begin{align*}
  D &\to \mathbb{B} \\
  z &\mapsto (0,z),
  \end{align*}
  it is immediate that any element $u \in \mathcal{F}_{\tilde E}$ satisfies $u = -\infty$ on $\{z_1 = 0\}\cap \mathbb{B}$. On the other hand, $\log|z_1| \in \mathcal{F}_{\tilde E}$, implying that in fact
  \[
  \bigcap_{u \in \mathcal{F}_{\tilde E}}\{u_*(z) = -\infty\} = \{z_1 = 0\}\cap \mathbb{B}.
  \]
  Hence, this example shows that the estimate given in Theorem~\ref{maintheorem} is sharp.
\end{example}
\section{Continuity of envelopes on Reinhardt domains}
Let \(A \subset \C^n\) be a set invariant under the toric action on \(\C^n\). We say that a function \(F:A \rightarrow \R\) is \textit{torically uniformly continuous} if the family \(\{F_z \suchthat z \in A\}\) is equicontinuous, where
\begin{align*}
F_z:S^1 \times \cdots \times S^1 &\to \R \\
(e^{i\theta_1}, \dots, e^{i\theta_n}) &\mapsto F(e^{i\theta_1}z_1, \dots, e^{i\theta_n}z_n).
\end{align*}
We denote the set of all such functions on \(A\) by \(T(A)\). Clearly \(T(A)\) is a vector space containing all toric functions, as well as all uniformly continuous functions if $A$ is bounded.
\begin{theorem}\label{thm:torisk}
Let \(\Omega \subset \C^n\) be a Reinhardt domain, and suppose that \(F\in \usc(\Omega)\cap T(\Omega)\) is bounded from below. Then \(P(F)\) is continuous.
\end{theorem}
\begin{proof}
Fix \(\varepsilon=\varepsilon_0/2> \delta >0\) such that 
\[
|\theta - \theta'| < \delta \implies |F_z(\theta) - F_z(\theta')| < \varepsilon
\]
for all \(z \in \Omega\). Since \(P(F)\) is bounded from below, it follows from the proof of \cite[Theorem~4]{fornaess} that there exist \(\nu_k, \varepsilon_{j,k}<\delta\), \(j =1,...,n\) such that \(\nu_k, \varepsilon_{j,k} \searrow 0\) as \(k\rightarrow \infty\) and continuous plurisubharmonic functions
\[
u_k : = \nu_k + \sup_{m>k} \Big( \frac{1}{2^n \varepsilon_{1,m}...\varepsilon_{n,m}}\int_{-\varepsilon_{1,m}}^{\varepsilon_{1,m}}...\int_{-\varepsilon_{n,m}}^{\varepsilon_{n,m}} u(e^{i\theta_1}z_1,...,e^{i\theta_n}z_n) \ d\theta_1...d\theta_n\Big)
\]
such that \(u_k \searrow P(F)\) pointwise on \(\Omega\). Since
\begin{align*}
u_k &\leq \nu_k + \sup_{m>k} \Big( \frac{1}{2^n \varepsilon_{1,m}...\varepsilon_{n,m}}\int_{-\varepsilon_{1,m}}^{\varepsilon_{1,m}}...\int_{-\varepsilon_{n,m}}^{\varepsilon_{n,m}} F(e^{i\theta_1}z_1,...,e^{i\theta_n}z_n) \ d\theta_1...d\theta_n\Big)\\
&\leq \delta +  F(z)+ \varepsilon \leq F(z) + \varepsilon_0,
\end{align*}
it follows that 
\[
u_k - \varepsilon_0  \leq F,
\]
which implies that is enough to construct the envelope over continuous functions, and so \(P(F)\) is lower semicontinuous. On the other hand, \(P(F)\) is upper semicontinuous by the Brelot--Cartan theorem, which concludes the proof.
\end{proof}
\begin{corollary}
Let \(\Omega\) be a Reinhardt domain, and suppose that \(u\in\psh(\Omega)\) is bounded from below. Then \(u\in T(\Omega) \implies u \in C(\Omega)\).
\end{corollary}
\begin{corollary}\label{walshsats}
Let \(\Omega\) be a bounded Reinhardt domain. Then
\[
\varphi \text{ uniformly continuous on }\Omega \implies P(\varphi) \in C(\Omega).
\]
\end{corollary}
\begin{remark}
Corollary~\ref{walshsats} extends to all unbounded Reinhardt domains where uniform continuity implies torically uniform continuity, under the additional assumption that \(\varphi\) is bounded from below. 
\end{remark}
We are only partially able to extend this result to envelopes for which the complex Monge--Ampère measures of the defining family are restricted by a continuously compliant measure. Specifically, we only consider envelopes of the form
\[P(F, f) := \sup \{u(z) \suchthat u \in \psh(\Omega), (dd^cu)^n \geq f \beta^n, u \leq F\},
\]
where \(f\geq 0\) and \(\beta^n\) denotes the volume measure. For \(f\in C(\bar \Omega)\), this may be done using the viscosity point of view. Let \(H_n^+\) denote the set of all semi-positive Hermitian \(n\) by \(n\) matrices and let \(\dot{H}_n^+ \subset H_n^+\) denote the matrices with determinant \(n^{-n}\). For \(H \in  \dot{H}_n^+\), we denote the corresponding Laplacian by
\[
\Delta_H := \sum_{i,j=1}^n h_{i,j}\frac{\partial^2}{\partial z_i \partial \bar{z}_j}.
\]
The proof relies on the fact that for \(u \in \psh(\Omega) \cap L^\infty_{loc}(\Omega)\), \(0 \leq f \in C(\Omega)\), 
\[
(dd^cu)^n \geq f \beta^n \iff \Delta_H u \geq f^{1/n} \text{ for all \(H \in \dot{H}_n^+\).}
\]
For more details on this technique, see the recent textbook by Guedj and Zeriahi~\cite[Section~5.2.2]{guedj}.
\begin{theorem}\label{thm:toriskmellan}
Let \(\Omega \subset \C^n\) be a bounded Reinhardt domain, \(0 \leq f \in C(\bar \Omega)\) and suppose that \(F\in \usc(\Omega)\cap T(\Omega)\) is bounded from below. Then \(P(F, f)\) is continuous.
\end{theorem}
\begin{proof}
Let \(u_k\) be as in the proof of Theorem~~\ref{thm:torisk}, with \(\delta\) small enough such that
\[
\frac{1}{2^n \varepsilon_{1,m}...\varepsilon_{n,m}}\int_{-\varepsilon_{1,m}}^{\varepsilon_{1,m}}...\int_{-\varepsilon_{n,m}}^{\varepsilon_{n,m}} f^{1/n}(e^{i\theta_1}z_1,...,e^{i\theta_n}z_n) \ d\theta_1...d\theta_n \geq f^{1/n}(z) - \varepsilon,
\]
and let
\[
u^m : = \frac{1}{2^n \varepsilon_{1,m}...\varepsilon_{n,m}}\int_{-\varepsilon_{1,m}}^{\varepsilon_{1,m}}...\int_{-\varepsilon_{n,m}}^{\varepsilon_{n,m}} u(e^{i\theta_1}z_1,...,e^{i\theta_n}z_n) \ d\theta_1...d\theta_n.
\]
For a positive test function \(\varphi\), we have
\begin{align*}
    &\langle \Delta_H u^m, \varphi \rangle = \int_\Omega u^m \Delta_H \varphi \ dz_1...dz_n\\
    &=\frac{1}{2^n \varepsilon_{1,m}...\varepsilon_{n,m}} \int_\Omega \int_{-\varepsilon_{1,m}}^{\varepsilon_{1,m}}...\int_{-\varepsilon_{n,m}}^{\varepsilon_{n,m}} u(e^{i\theta_1}z_1,...,e^{i\theta_n}z_n) \Delta_H \varphi \ d\theta_1...d\theta_n  dz_1...dz_n \\
    &=\frac{1}{2^n \varepsilon_{1,m}...\varepsilon_{n,m}} \int_{-\varepsilon_{1,m}}^{\varepsilon_{1,m}}...\int_{-\varepsilon_{n,m}}^{\varepsilon_{n,m}}\int_\Omega  u(e^{i\theta_1}z_1,...,e^{i\theta_n}z_n) \Delta_H \varphi \  dz_1...dz_n  d\theta_1...d\theta_n \\
    &\geq\frac{1}{2^n \varepsilon_{1,m}...\varepsilon_{n,m}} \int_{-\varepsilon_{1,m}}^{\varepsilon_{1,m}}...\int_{-\varepsilon_{n,m}}^{\varepsilon_{n,m}}\int_\Omega  f^{1/n}(e^{i\theta_1}z_1,...,e^{i\theta_n}z_n) \varphi \  dz_1...dz_n  d\theta_1...d\theta_n \\
    &\geq \langle f^{1/n}(z) - \varepsilon, \varphi \rangle
\end{align*}
using Fubini's theorem. In particular 
\[
\Delta_H (u^m + \varepsilon (|z|^2 -K)) \geq f^{1/n},
\]
where \(K\) is chosen such that \(|z|^2 -K \leq 0\). By the monotonicity of the complex Monge--Ampère operator, 
\[
(dd^c u_k + \varepsilon (|z|^2 -K))^n \geq f,
\]
and reasoning as in the proof of Theorem~\ref{thm:torisk}, we conclude that the envelope is continuous.  
\end{proof}
\begin{remark}
Using the standard balayage argument, one may show that \(F\) being harmonic ensures that the prescribed complex Monge--Ampère measure \(f\beta^n\) is attained by \(P(F,f)\).
\end{remark}
We end this section by extending this further, under the assumptions that \(F\) is harmonic and \(\Omega \subset \C^n\) is a bounded, strictly pseudoconvex Reinhardt domain. The arguments needed are due to Ko\l{}odziej~\cite{kolod1, kolod2}, in particular his proof that the complex Monge--Ampère equation has continuous solutions for densities of the form \(f \beta^n\), where \(0 \leq f \in L^p(\Omega)\), \(p>1\). Although his methods allow for a larger class of measures, we will for simplicity settle with proving the analogous result in our setting. Let 
\begin{align*}
    h: \R_+ & \rightarrow (1, \infty) \\
    t&\mapsto (1+ \log(t+1))^{n+1}
\end{align*}
and define
\[
\mathcal{F}(A,h, \Omega):= \Big\{\mu \suchthat \text{ for all compact }K\in \Omega, \mu(K) \leq \frac{A \cdot \text{cap}(K, \Omega)}{h((\text{cap}(K, \Omega)^{-1/n})}\Big\},
\]
where \(A>0\) and \(\text{cap}(K, \Omega)\) denotes the relative capacity
\[
\text{cap}(K, \Omega) := \sup \{ \int_K (dd^cu)^n \suchthat u\in \psh(\Omega), -1 \leq u < 0\}.
\]
The reason for the rather complicated construction of \(\mathcal{F}(A,h, \Omega)\) is the following result, which constitutes one of the main ingredients in Ko\l{}odziej's method.
\begin{lemma}\label{kolodlemma}
Let \(\Omega\) be a strictly pseudoconvex domain, and fix \(A>0\). Then there exists an increasing function \(\kappa: \R_+ \rightarrow \R_+\)  with the following properties: 

\begin{itemize}
    \item For all \(v \in \psh(\Omega)\cap C(\Omega)\) and \(u \in \psh (\Omega) \cap L^\infty (\Omega)\) such that \((dd^cu)^n \in \mathcal{F}(A,h, \Omega)\) and the set \(U(s) = \{u-s < v \}\) is nonempty and relatively compact in \(\Omega\) for \(s \in [S, S+D]\), we have
\[
D \leq \kappa(\text{\normalfont{cap}}(U(S+D), \Omega)).
\]
    \item \(\lim_{t\rightarrow 0}\kappa(t)=0\).
\end{itemize}
\end{lemma}
In relation to \(\mathcal{F}(A,h, \Omega)\), we define the set
\[
L^{\psi_h}(c_0, \Omega) := \{ f \in L^1(\Omega) \suchthat f \geq 0, \int_\Omega \psi_h (f) dV \leq c_0\},
\]
where \(\psi_h:\R_+ \rightarrow \R_+\) is defined by
\[
\psi_h (t) := t(\log (1+t))^n h(\log(1+t)).
\]
Clearly,
\begin{itemize}
    \item \(\frac{\psi_h(t)}{t}\) increases to \(\infty\) as \(t \rightarrow \infty\),
    \item for all \(p >1, L^p(\Omega) \subset \bigcup_{c>0} L^{\psi_h}(c, \Omega)\),
\end{itemize}
and less trivially, for each \(c_0>0\) we may find \(A>0\) such that the inclusion \(L^{\psi_h}(c_0, \Omega) \subset \mathcal{F}(A,h, \Omega)\) holds. See Ko\l{}odziej's book~\cite{kolod2} for more details.
\begin{theorem}\label{kolodtheorem}
Let \(\Omega \subset \C^n\) be a bounded, strictly pseudoconvex Reinhardt domain, and assume that \(F\in T(\Omega)\) is harmonic and bounded from below. Then for \(0 \leq f \in L^p(\Omega)\), \(p>1\),
\(P(F, f)\) is continuous.
\end{theorem}
\begin{proof}
Fix \(c_0\) and \(A\) such that \(f \in L^{\psi_h}(c_0, \Omega) \cap \mathcal{F}(A,h, \Omega)\). Using Ko\l{}odziej's original result, it follows that \(P(0,f)\) is continuous and that \(P(0,f)|_{\partial \Omega}=0\), so we may for each \(d>0\) find a strictly pseudoconvex domain \(\Omega' \Subset \Omega\) for which \(P(0,f) + d > 0\) on \(\bd \Omega'\). Let
\begin{align*}
u_r &: =  P(F,f) \ast \rho_r \\
v_r &: = P(F,0) \ast \rho_r
\end{align*}
denote convolution with standard radial mollifiers. Since \(P(F,0)\) is continuous by Theorem~\ref{thm:torisk}, Dini's theorem implies that \(v_r\) converge locally uniformly to \(P(F,0)\), and so we may find \(d> \varepsilon_r \searrow 0\) such that
\[
u_r  \leq v_r  \leq P(F,0) + \varepsilon_r \leq P(F,0) + P(0,f) +2d \leq P(F,f) + 2d
\]
on \(\partial \Omega'\). Now note that
\[
\{u_r > P(F, f) + td\}\cap \Omega'
\]
is relatively compact in \(\Omega'\) for \(4\geq t \geq3\). In order to reach a contradiction, assume that these sets are nonempty. Then, as \(F\) is harmonic, the balayage procedure implies that
\[
(dd^cP(F,f))^n =f \in \mathcal{F}(A', h, \Omega')
\]
for some \(A'>0\), with the consequence that 
\[
\kappa(\text{cap}(\{ P(F, f) + 3d < u_r\}, \Omega')) \geq d, 
\]
using Lemma~\ref{kolodlemma} with \(S = -4d, D=d\) and the fact that \(P(F,f) \in L^\infty(\Omega')\). We conclude that the relative capacity is bounded away from zero. This is impossible since \(u_r \searrow P(F,f)\), and in particular converge in capacity \cite[Corollary~1.2.10]{kolod1}.
\end{proof}
\section{Large discontinuity sets in the unit ball}
In this section, we consider the Dirichlet problem for the complex Monge--Ampère equation in the unit ball \(B \subset \C^n\) for a class of the boundary data where the discontinuity set is not b-pluripolar. Specifically, our boundary function will be a characteristic function 
\[
\phi_A(z) =
        \begin{cases}
            -1   &z \in A \\
            0   &z\in \partial B \setminus A
        \end{cases}
\]
for a multi-circular, open set \(A \subset \partial B\) such that \(\bar A\) does not meet the hyperplanes \(\{z_j=0\}\). Since the relative boundary \(\partial A\) is multi-circular as well, \cite[Example~3.4]{djire} shows that the discontinuity set of \(\phi_A\) is not b-pluripolar.

The tool needed to prove uniqueness in this setting is provided by the following lemma. 
\begin{lemma}\label{contapprox}
Let \(A \subset \partial B\) satisfy the requirements above. Then \(P(\phi_{A})\) is continuous, and may be written as an envelope over uniformly continuous functions. 
\end{lemma}
\begin{proof}
We begin by introducing the notation
\begin{align*}
    P_{\text{cont}}(\phi_{A})  &:= \sup \{u(z) \suchthat u \in \psh(B)\cap C(\bar B), u \leq \phi_{A}\} \\
    P_{\text{rad}}(\phi_{A})  &:= \sup \{u(z) \suchthat u \in \psh(B), \limsup_{r\rightarrow 1}u(r\zeta) \leq \phi_{A}(\zeta) \text{ for all }\zeta \in \partial B\}.
\end{align*}
As a direct consequence of \cite[Theorem~2.11]{djire}, \(P_{\text{cont}}(\phi_{\bar A})=P(\phi_{\bar A})=P_{\text{rad}}(\phi_{\bar A})\), and clearly
\[
P_{\text{cont}}(\phi_{\bar A}) \leq P_{\text{cont}}(\phi_{A}) \leq P(\phi_{A}) \leq P_{\text{rad}}(\phi_{A}).
\]
We claim that \(P_{\text{rad}}(\phi_{\bar A}) = P_{\text{rad}}(\phi_{A})\). To see this, fix 
\[
\zeta = (\zeta_1, ..., \zeta_n) \in \partial A,
\]
an element \(\varphi\) in the defining family for \(P_{\text{rad}}(\phi_{A})\), and a point \(z_0\) in the inward normal \(n_\zeta\) to \(\bd B\) at \(\zeta\). Then, \(z_0\) necessarily lies in all polydisks 
\[
D(0,r)=D(0, r_1) \times \dots \times D(0, r_n)
\]
such that \(\sum r_i^2 = 1\) and with \(r_i\) sufficiently close to \(|\zeta_i|\). In particular, we may find a polydisk containing \(z_0\) whose distinguished boundary is contained in \(A\), which implies that \(\varphi(z_0) \leq -1\). Since this holds for all \(z_0, \zeta\) and \(\varphi\), the defining families for \(P_{\text{rad}}(\phi_{\bar A})\) and \(P_{\text{rad}}(\phi_{A})\) must coincide. We conclude that
\[
P(\phi_{A}) = P_{\text{cont}}(\phi_{A}),
\]
and that \(P(\phi_{A})\) is continuous. 
\end{proof}
We are now ready to prove that b-pluripolar discontinuity or continuity almost everywhere are in general not necessary conditions in order to provide uniqueness and continuity of solutions to the complex Monge--Ampère equation, if one additionally requires that lower limits of the solution coincides with the lower limits of the boundary function.
\begin{theorem}\label{sista}
Let \(A \subset \partial B\) satisfy the requirements above, and let \(\mu\) be a compliant measure. Then the Dirichlet problem
    \[\begin{cases}
    u \in \psh(\Omega)\cap L^\infty(B)  \\
    (dd^cu)^n= \mu \\
        \lim_{B \ni \zeta \rightarrow z_0} u(\zeta) = \phi_A(z_0), \quad &\forall z_0 \in \bd B \setminus \partial A \\
                        \liminf_{B \ni \zeta \rightarrow z_0} u(\zeta) \geq -1,\quad &\forall z_0 \in \partial A
    \end{cases}\]
has a unique solution. This solution is furthermore continuous on \(B\) if \(\mu\) is a continuously compliant measure of the form \(\mu = f \beta^n\), where  \(0 \leq f \in L^p(\Omega)\), and \(p>1\).
\end{theorem}
\begin{proof}
The proof of Theorem~\ref{maintheorem} implies that \(P(\phi_A, \mu)\) is a solution. To prove that this envelope solves the Dirichlet problem uniquely, suppose that \(\tilde u\) is another solution. Then \(\tilde u \leq P(\phi_A, \mu)\), since otherwise there must exist \(z_0 \in \bd A\) such that
\[
\limsup_{B \ni \zeta \rightarrow z_0} \tilde u(z_0) >  \phi_A(z_0) = 0,
\]
which one may show is impossible using polydisks as in the proof of Lemma~\ref{contapprox}.

Now pick any member \(v\) in the defining family for \(P_{\text{cont}}(\phi_{A})\). Clearly
\[
\limsup_{z \rightarrow \bd B}(v(z)+P(0, \mu)(z) - \tilde u(z)) \leq 0,
\]
and so by the domination principle, \(\tilde u \geq v + P(0, \mu)\). Since this holds for all \(v\), it follows from Lemma~\ref{contapprox} that
\[
\tilde u \geq P(\phi_{A}) + P(0, \mu)
\]
on \(B\). However,
\[
P(\phi_{A}) + P(0, \mu) \leq P(\phi_A, \mu) \leq P(\phi_{A}),
\]
and \(\lim_{z \rightarrow \bd B} P(0, \mu)(z) = 0\) since \(\mu\) is compliant. Hence
\begin{align*}
   \limsup_{z \rightarrow \bd B}(P(\phi_A, \mu)(z) - \tilde u(z)) &\leq \limsup_{z \rightarrow \bd B}(P(\phi_{A})(z) - \tilde u(z)) \\
   &= \limsup_{z \rightarrow\bd B}(P(\phi_{A})(z) + P(0, \mu)(z) - \tilde u(z)) \leq 0,
\end{align*}
and so by the domination principle, \(P(\phi_A, \mu) =\tilde u\).

Now assume that \(\mu = f \beta^n\), where  \(0 \leq f \in L^p(\Omega)\), and \(p>1\). Since \(\phi_A\) is upper semicontinuous, we may find toric functions \(\phi_k \in C(\bd \Omega)\) such that \(\phi_k \searrow \phi_A\), and harmonic functions \(h_k\) defined by the property
\[
\lim_{B \ni z \rightarrow \zeta} h_k(z) = \phi_k(\zeta)
\]
for \(\zeta \in \bd B\). As all toric actions preserve these boundary values, 
\[
h_k(z_1, ..., z_n) = h_k(|z_1|, ..., |z_n|),
\]
implying that \(h_k\) converges to a harmonic function \(h_A\in T(B)\) by Harnack's theorem. Comparing defining families, it is clear that
\[
P(h_A, f) \leq P(\phi_A, f) \leq P(h_A, f)
\]
and hence \(P(\phi_A, f) = P(h_A, f)\). Theorem~\ref{kolodtheorem} then yields continuity on \(B\).
\end{proof}
\begin{corollary}\label{sistacor}
In the class of bounded, positive plurisubharmonic functions in the unit ball \(B\) in \(\C^3\), there exists an element that is uniquely determined by its complex Monge--Ampère measure and its boundary behavior at a set of arbitrarily small Lebesgue measure.
\end{corollary}
\begin{proof}
Note that in the Reinhardt diagram of \(B\), where we consider two points \(z, w\) to be equivalent if 
\[
(|z_1|, |z_2|, |z_3|) = (|w_1|, |w_2|, |w_3|),
\]
one may identify the part of the boundary that does not meet the hyperplanes \(\{z_j=0\}\) with a bounded open subset \(U\) of \(\R^2\) by projection to the \((|z_1|,|z_2|)\)-plane. Using the Smith–Volterra–Cantor construction, we may then find a compact, totally disconnected set \(C \subset U\), carrying an arbitrarily large proportion of the total Lebesgue measure. By the Denjoy–Riesz theorem, there exists a Jordan curve containing \(C\), which divides \(U\) into an interior part and an exterior part by the Jordan curve theorem. Hence, mapping the interior set back to the surface of \(B\) produces a set \(A\) satisfying the requirements of Theorem~\ref{sista}, with Lebesgue measure of \(\partial B \setminus \bd A\) arbitrarily small. The function $P(\phi_A, 0) + 1$ is then uniquely determined (among all bounded, positive plurisubharmonic functions) by its
complex Monge--Ampère measure and its boundary behavior at \(\partial B \setminus \bd A\).
\end{proof}


\begin{thebibliography}{99}
    \bibitem{bedford}
    E. Bedford, \emph{An example of a plurisubharmonic measure on the unit ball in \(\C^2\)}, Michigan Math. J. 27 (1980), no. 3, pp.~365--370.
    
    \bibitem{djire}
    I. K. Djire, J. Wiegerinck, \emph{Characterizations of boundary pluripolar hulls}, 
    Complex Variables and Elliptic Equations, 61(8) (2016), pp.~1133--1144.

    \bibitem{fatou}
     P. Fatou, \emph{Series trigonométriques et séries de Taylor}, Acta Math., 30 (1906), pp.~335--400.

    \bibitem{fornaess}
     J. E. Fornæss and J. Wiegerinck, \emph{Approximation of plurisubharmonic functions}, Ark. Mat. 27 (1989), pp.~257–-272.
        
    \bibitem{garnett}
    J. B. Garnett and D. E. Marshall, \textit{Harmonic measure}, New Mathematical Monographs, vol. 2, Cambridge University Press, Cambridge (2005).
    
    \bibitem{guedj}
    V. Guedj, A. Zeriahi, \textit{Degenerate Complex Monge--Ampère Equations}, EMS (2017).

    \bibitem{kolod1}
    S. Ko\l{}odziej, \emph{The complex Monge--Ampère equation}, Acta Math., 180(1) (1998), pp.~69--117.

    \bibitem{kolod2}
    S. Ko\l{}odziej, \textit{The Complex Monge--Ampère Equation and Pluripotential Theory}, Memoirs Amer. Math. Soc. 178 (2005).


    \bibitem{nilsson}
    M. Nilsson, \emph{Continuity of envelopes of unbounded plurisubharmonic functions}, Math. Z. 301 (2022), pp.~3959–-3971. 

    \bibitem{ransford}
    T. Ransford, \textit{Potential Theory in the Complex Plane}, London Math. Soc. Student Texts No. 28 (Cambridge University Press, 1995).

    \bibitem{rashkovskii}
    A. Rashkovskii, \emph{Rooftop envelopes and residual plurisubharmonic functions}, Annales Polonici Mathematici 128 (2022), pp.~159--191.

    \bibitem{sadu}
    A. Sadullaev, \emph{Plurisubharmonic measures and capacities on complex manifolds}, Russ. Math. Surv. 36, No. 4 (1981), pp.~61--119.

    \bibitem{tong}
    H. Tong, \emph{Some characterizations of normal and perfectly normal spaces}, Duke Math. J. 19, No. 2 (1952), pp.~289--292.    
    \bibitem{walsh}
    J.B. Walsh, \emph{Continuity of envelopes of plurisubharmonic functions},
    J. Math. Mech. 18 (1968/69), pp.~143--148.

\end{thebibliography}
\end{document}